\documentclass{amsart}
\usepackage[dvips]{graphicx}
\usepackage{amssymb}

\newtheorem{theorem}{Theorem}

\newtheorem{lemma}[theorem]{Lemma}
\newtheorem{proposition}[theorem]{Proposition}

\theoremstyle{remark}

\newtheorem{remark}[theorem]{\sc Remark}
\newtheorem{example}[theorem]{\sc Example}
\newtheorem*{acknowledgment}{\sc Acknowledgment}

\numberwithin{equation}{section}

\begin{document}

\title{Toric Fano varieties associated to graph cubeahedra}

\author{Yusuke Suyama}
\address{Department of Mathematics, Graduate School of Science, Osaka University,
1-1 Machikaneyama-cho, Toyonaka, Osaka 560-0043 JAPAN}
\email{y-suyama@cr.math.sci.osaka-u.ac.jp}



\date{\today}


\begin{abstract}
We give a necessary and sufficient condition for the nonsingular projective toric variety
associated to the graph cubeahedron of a finite simple graph
to be Fano or weak Fano in terms of the graph.
\end{abstract}

\maketitle

\section{Introduction}

An $n$-dimensional {\it toric variety} is a normal algebraic variety $X$ over $\mathbb{C}$
containing the algebraic torus $(\mathbb{C}^*)^n$ as an open dense subset,
such that the natural action of $(\mathbb{C}^*)^n$ on itself extends to an action on $X$,
where $\mathbb{C}^*=\mathbb{C} \setminus \{0\}$.
It is well known that the category of toric varieties is equivalent to the category of fans.

An $n$-dimensional simple convex polytope in $\mathbb{R}^n$
is called a {\it Delzant polytope} if for every vertex,
the outward-pointing primitive normal vectors of the facets containing the vertex
form a basis for $\mathbb{Z}^n$.
The toric variety corresponding to the normal fan of a Delzant polytope
is nonsingular and projective.
Carr and Devadoss \cite{CD} introduced {\it graph associahedra} for finite simple graphs.
A graph associahedron can be realized as a Delzant polytope in a canonical way.
Graph associahedra include many important families of polytopes
such as associahedra (or Stasheff polytopes),
cyclohedra (or Bott--Taubes polytopes), stellohedra and permutohedra,
and toric varieties associated to graph associahedra
are special cases of wonderful models of subspace arrangements
introduced by De Concini and Procesi \cite{DP}.
On the other hand,
Devadoss--Heath--Vipismakul \cite{DHV} introduced {\it graph cubeahedra} for graphs
and showed that graph associahedra and graph cubeahedra
appear as some compactified moduli spaces of marked bordered Riemann surfaces.
A graph cubeahedron can also be realized as a Delzant polytope.

It is natural to ask how geometric properties of the toric variety
associated to a graph associahedron or a graph cubeahedron
translate into properties of the graph.
The rational Betti numbers of the real toric manifold,
the set of real points in the associated toric variety,
are computed in \cite{CP} for a graph associahedron,
and in \cite{PPP} for a graph cubeahedron.
A nonsingular projective variety is said to be {\it Fano} (resp.\ {\it weak Fano})
if its anticanonical divisor is ample (resp.\ nef and big).
The author \cite{Suyama} characterized finite simple graphs
whose graph associahedra yield toric Fano or toric weak Fano varieties.
In this paper, we give a necessary and sufficient condition
for the toric variety associated to the graph cubeahedron
to be Fano or weak Fano in terms of the graph
(see Theorems \ref{main1} and \ref{main2}).
The proofs are purely combinatorial.

The structure of the paper is as follows:
In Section 2, we review the construction of a graph cubeahedron
and describe its normal fan.
In Section 3, we characterize finite simple graphs
whose graph cubeahedra yield toric Fano or toric weak Fano varieties,
and give a few examples.
We prove our theorems in Section 4.
In this paper, we denote by $X(\Delta)$ the toric variety corresponding to a fan $\Delta$,
and $\Delta_P$ the normal fan of a Delzant polytope $P$.

\begin{acknowledgment}
This work was supported by Grant-in-Aid for JSPS Fellows 15J01000, 18J00022.
The author wishes to thank Professors Mikiya Masuda and Seonjeong Park
for their valuable comments.
\end{acknowledgment}

\section{Toric varieties associated to graph cubeahedra}

A {\it finite simple graph} is a finite graph $G$
on the node set $V(G)=\{1, \ldots, n\}$ with no loops and no multiple edges.
We denote by $E(G)$ its edge set.
For $I \subset V(G)$, the {\it induced subgraph} $G|_I$ is defined
by $V(G|_I)=I$ and $E(G|_I)=\{\{v, w\} \in E(G) \mid v, w \in I\}$.

Let $\Box^n$ be the standard $n$-dimensional cube whose facets are labeled by
$1, \ldots, n$ and $\overline{1}, \ldots, \overline{n}$,
where the two facets labeled by $i$ and $\overline{i}$ are on opposite sides.
Every face of $\Box^n$ is labeled by a subset
$I \subset \{1, \ldots, n, \overline{1}, \ldots, \overline{n}\}$ such that
$I \cap \{1, \ldots, n\}$ and $\{i \in \{1, \ldots, n\} \mid \overline{i} \in I\}$ are disjoint.
The face corresponding to $I$ is the intersection of the facets
labeled by the elements of $I$.
Let $\mathcal{I}_G=\{I \subset V(G) \mid G|_I \mbox{ is connected}, I \ne \emptyset\}$.
The {\it graph cubeahedron} $\Box_G$ is obtained from $\Box^n$
by truncating the faces labeled by the elements of $\mathcal{I}_G$
in increasing order of dimension.
The following lemma implies that $\Box_G$ can be realized as a Delzant polytope.
In particular, the associated toric variety
$X(\Delta_{\Box_G})$ is nonsingular and projective.

\begin{lemma}[{\cite[Lemma 2.5]{CPP}}]
Let $P$ be a Delzant polytope and let $F$ be a face of codimension $\geq 2$ of $P$.
Then there exists a canonical truncation of $P$ along $F$
such that the result $\mathrm{Cut}_F(P)$ is also a Delzant polytope
and the associated toric variety $X(\Delta_{\mathrm{Cut}_F(P)})$
is the blow-up of $X(\Delta_P)$ along the subvariety corresponding to $F$.
\end{lemma}

\begin{example}
Let $P_n$ be a path with $n$ nodes.
Then the graph cubeahedra $\Box_{P_2}$ and $\Box_{P_3}$ are illustrated in Figure \ref{gcp}.
\begin{figure}[htbp]
\begin{center}
\includegraphics[width=3cm]{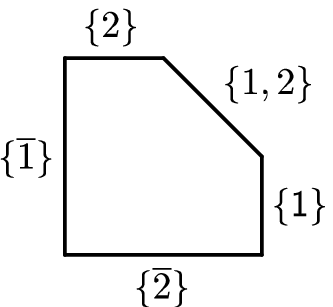}
\hspace{2cm}
\includegraphics[width=3cm]{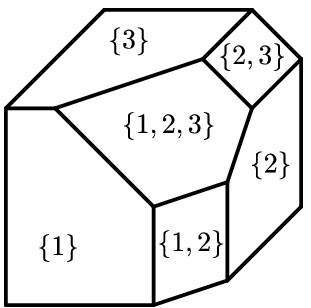}
\caption{the graph cubeahedra $\Box_{P_2}$ and $\Box_{P_3}$.}
\label{gcp}
\end{center}
\end{figure}
\end{example}

For a finite simple graph $G$ on $V(G)=\{1, \ldots, n\}$,
we have a bijection between the set of facets of $\Box_G$
and $\mathcal{I}_G \cup \{\{\overline{1}\}, \ldots, \{\overline{n}\}\}$.
We denote by $F_I$ the facet corresponding to
$I \in \mathcal{I}_G \cup \{\{\overline{1}\}, \ldots, \{\overline{n}\}\}$.
The outward-pointing primitive normal vector $e_I$ of $F_I$ is given by
\begin{equation*}
e_I=\left\{\begin{array}{ll}
\sum_{i \in I}e_i & (I \in \mathcal{I}_G), \\
-e_i & (I=\{\overline{i}\}, i \in \{1, \ldots, n\}). \end{array}\right.
\end{equation*}

\begin{theorem}[{\cite[Theorem 12]{DHV}}]
Let $G$ be a finite simple graph.
Then the two facets $F_I$ and $F_J$ of the graph cubeahedron $\Box_G$ intersect
if and only if one of the following holds:
\begin{enumerate}
\item $I, J \in \mathcal{I}_G$ and we have either
$I \subset J$ or $J \subset I$ or $I \cup J \notin \mathcal{I}_G$.
\item One of $I$ and $J$, say $I$, is in $\mathcal{I}_G$ and $J=\{\overline{j}\}$
for some $j \in \{1, \ldots, n\} \setminus I$.
\item $I=\{\overline{i}\}$ and $J=\{\overline{j}\}$ for some $i, j \in \{1, \ldots, n\}$.
\end{enumerate}
Furthermore, $\Box_G$ is a flag polytope, that is,
any collection of its pairwise intersecting facets has a nonempty intersection.
\end{theorem}

We describe the normal fan $\Delta_{\Box_G}$ of $\Box_G$ explicitly. Let
\begin{equation*}
\mathcal{N}^\Box(G)=\{N \subset \mathcal{I}_G
\cup \{\{\overline{1}\}, \ldots, \{\overline{n}\}\} \mid
F_I \mbox{ and } F_J \mbox{ intersect for any } I, J \in N\}.
\end{equation*}
For $N \in \mathcal{N}^\Box(G) \setminus \{\emptyset\}$,
we denote by $\sigma_N$ the $|N|$-dimensional cone
$\sum_{I \in N}\mathbb{R}_{\geq 0}e_I$ in $\mathbb{R}^n$,
where $\mathbb{R}_{\geq 0}$ is the set of nonnegative real numbers,
and we define $\sigma_\emptyset$ to be $\{0\} \subset \mathbb{R}^n$.
Then $\{\sigma_N \mid N \in \mathcal{N}^\Box(G)\}$
is the normal fan $\Delta_{\Box_G}$.
Note that if $G_1, \ldots, G_m$ are the connected components of $G$,
then $X(\Delta_{\Box_G})$ is isomorphic to the product
$X(\Delta_{\Box_{G_1}}) \times \cdots \times X(\Delta_{\Box_{G_m}})$.

\begin{example}\label{p2}
The normal fan of the graph cubeahedron $\Box_{P_2}$ of a path $P_2$
is illustrated in Figure \ref{xp2} and thus the associated toric variety
is $\mathbb{P}^1 \times \mathbb{P}^1$ blown-up at one point.
\begin{figure}[htbp]
\begin{center}
\includegraphics[width=6cm]{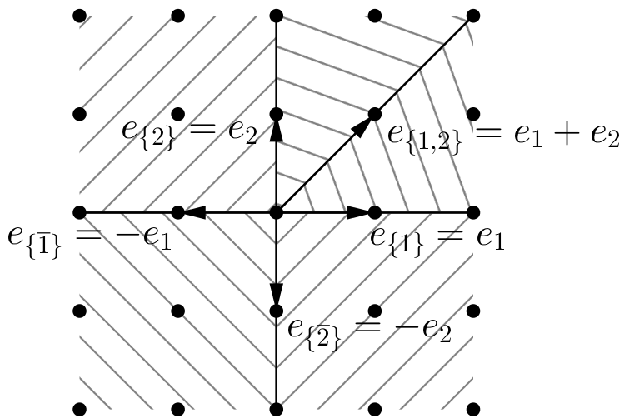}
\caption{the normal fan of the graph cubeahedron $\Box_{P_2}$.}
\label{xp2}
\end{center}
\end{figure}
\end{example}

\section{Main theorems}

Our main results are the following:

\begin{theorem}\label{main1}
Let $G$ be a finite simple graph. Then the following are equivalent:
\begin{enumerate}
\item The nonsingular projective toric variety $X(\Delta_{\Box_G})$
associated to the graph cubeahedron $\Box_G$ is Fano.
\item Every connected component of $G$ has at most two nodes.
\end{enumerate}
In particular, any toric Fano variety associated to a graph cubeahedron
is a product of copies of $\mathbb{P}^1$
and $\mathbb{P}^1 \times \mathbb{P}^1$ blown-up at one point.
\end{theorem}

\begin{theorem}\label{main2}
Let $G$ be a finite simple graph. Then the following are equivalent:
\begin{enumerate}
\item The toric variety $X(\Delta_{\Box_G})$ is weak Fano.
\item For any subset $I$ of $V(G)$,
the induced subgraph $G|_I$ is not isomorphic to any of the following:
\begin{enumerate}
\item A cycle with $\geq 4$ nodes.
\item A diamond graph, that is,
the graph obtained by removing an edge from a complete graph with four nodes
(see Figure \ref{diamond}).
\item A claw, that is, a star with three edges.
\end{enumerate}
\end{enumerate}
In particular, if $X(\Delta_{\Box_G})$ is weak Fano, then $G$ is chordal.
\end{theorem}

\begin{figure}[htbp]
\begin{center}
\includegraphics[width=3cm]{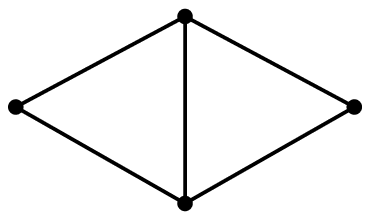}
\caption{a diamond graph.}
\label{diamond}
\end{center}
\end{figure}

\begin{example}
\begin{enumerate}
\item If $G$ is a path, then the toric variety $X(\Delta_{\Box_G})$ is weak Fano.
\item The {\it graph associahedron} of a graph is obtained from a product of simplices
by truncating the faces corresponding to proper connected induced subgraphs
of each connected component of the graph.
In a previous paper \cite{Suyama}, the author proved
that the toric variety associated to the graph associahedron is weak Fano
if and only if every connected component of the graph
does not have a cycle with $\geq 4$ nodes or a diamond graph as a proper induced subgraph.
In particular, a star yields a toric weak Fano variety.
The same conclusion follows from Theorem \ref{main2}.
Manneville and Pilaud showed
that the toric variety associated to the graph associahedron of a star with $n+1$ nodes
is isomorphic to that associated to the graph cubeahedron
of a complete graph with $n$ nodes (see \cite[Example 62 (i)]{MP}).
Theorem \ref{main2} implies that it is weak Fano.
\item If $G$ is a graph obtained by connecting more than two graphs with one node,
then $X(\Delta_{\Box_G})$ is not weak Fano.
\item The toric variety associated to the graph cubeahedron
of the graph in Figure \ref{ex} is weak Fano.
\begin{figure}[htbp]
\begin{center}
\includegraphics[width=5cm]{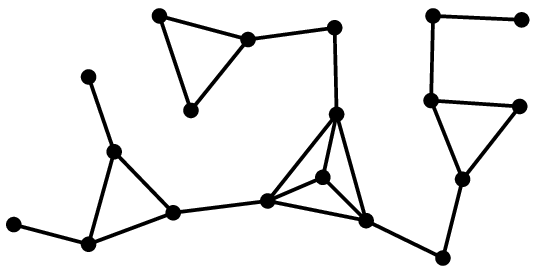}
\caption{an example.}
\label{ex}
\end{center}
\end{figure}
\end{enumerate}
\end{example}

\begin{remark}
Manneville and Pilaud proved
that for connected graphs $G$ and $G'$,
the graph associahedron of $G$ and the graph cubeahedron of $G'$
are combinatorially equivalent
if and only if $G$ is a tree with at most one node of degree more than two
and $G'$ is its line graph (see \cite[Proposition 64]{MP}).
Hence we see that there exist many toric varieties associated to graph cubeahedra
that are not associated to graph associahedra.
\end{remark}

\section{Proofs of main theorems}

First we recall a description of the intersection number of the anticanonical divisor
with a torus-invariant curve, see \cite{Oda} for details.
Let $\Delta$ be a nonsingular complete fan in $\mathbb{R}^n$.
For $r=0, 1, \ldots, n$,
we denote by $\Delta(r)$ the set of $r$-dimensional cones in $\Delta$.
For an $(n-1)$-dimensional cone $\tau$ in $\Delta$,
the intersection number of the anticanonical divisor $-K_{X(\Delta)}$ with
the torus-invariant curve $V(\tau)$ corresponding to $\tau$
can be computed as follows:

\begin{proposition}\label{intersection number}
Let $\Delta$ be a nonsingular complete fan in $\mathbb{R}^n$
and $\tau=\mathbb{R}_{\geq 0}v_1+\cdots+\mathbb{R}_{\geq 0}v_{n-1} \in \Delta(n-1)$,
where $v_1, \ldots, v_{n-1}$ are primitive vectors in $\mathbb{Z}^n$.
Let $v$ and $v'$ be the distinct primitive vectors in $\mathbb{Z}^n$ such that
$\tau+\mathbb{R}_{\geq 0}v$ and $\tau+\mathbb{R}_{\geq 0}v'$ are in $\Delta(n)$.
Then there exist unique integers $a_1, \ldots, a_{n-1}$ such that
$v+v'+a_1v_1+\cdots+a_{n-1}v_{n-1}=0$.
Furthermore, the intersection number $(-K_{X(\Delta)}.V(\tau))$ is equal to
$2+a_1+\cdots+a_{n-1}$.
\end{proposition}

\begin{proposition}\label{Nakai}
Let $X(\Delta)$ be an $n$-dimensional nonsingular projective toric variety.
Then the following hold:
\begin{enumerate}
\item $X(\Delta)$ is Fano if and only if
$(-K_{X(\Delta)}.V(\tau))$ is positive for every
$(n-1)$-dimensional cone $\tau$ in $\Delta$ (\cite[Lemma 2.20]{Oda}).
\item $X(\Delta)$ is weak Fano if and only if
$(-K_{X(\Delta)}.V(\tau))$ is nonnegative for every
$(n-1)$-dimensional cone $\tau$ in $\Delta$ (\cite[Proposition 6.17]{Sato}).
\end{enumerate}
\end{proposition}

We are now ready to prove Theorem \ref{main1}.

\begin{proof}[Proof of Theorem \ref{main1}]
$(1) \Rightarrow (2)$:
Suppose that $V(G)=\{1, \ldots, n\}$ and
$G$ has a connected component with $\geq 3$ nodes.
We may assume that $G|_{\{1, \ldots, k\}}$ is a connected component of $G$ with $k \geq 3$
and $G|_{\{1, \ldots, i\}}$ is connected for every $i=1, \ldots, k$.
We may further assume that $G|_{\{2, 3\}}$ is connected.
We consider
\begin{equation*}
N=\{\{2\}, \{1, 2, 3\}, \{1, 2, 3, 4\}, \ldots, \{1, \ldots, k\},
\{\overline{k+1}\}, \{\overline{k+2}\}, \ldots, \{\overline{n}\}\} \in \mathcal{N}^\Box(G).
\end{equation*}
Then $N \cup \{\{1, 2\}\}$ and $N \cup \{\{2, 3\}\}$
are maximal (by inclusion) elements of $\mathcal{N}^\Box(G)$. Since
\begin{equation*}
e_{\{1, 2\}}+e_{\{2, 3\}}-e_{\{2\}}-e_{\{1, 2, 3\}}=0,
\end{equation*}
Proposition \ref{intersection number} gives $(-K_{X(\Delta_{\Box_G})}.V(\sigma_N))=2-2=0$.
Therefore $X(\Delta_{\Box_G})$ is not Fano by Proposition \ref{Nakai} (1).

$(2) \Rightarrow (1)$:
If $G_1, \ldots, G_m$ are the connected components of $G$,
then $X(\Delta_{\Box_G})$ is isomorphic to
$X(\Delta_{\Box_{G_1}}) \times \cdots \times X(\Delta_{\Box_{G_m}})$.
Since the product of nonsingular projective toric varieties is Fano
if and only if every factor is Fano,
it suffices to show that $X(\Delta_{\Box_G})$ is Fano
if $G$ is connected and $|V(G)| \leq 2$.
If $G$ is a singleton graph, then $X(\Delta_{\Box_G})=\mathbb{P}^1$, which is Fano.
If $G$ is a path with two nodes,
then $X(\Delta_{\Box_G})$ is $\mathbb{P}^1 \times \mathbb{P}^1$ blown-up at one point,
which is Fano (see Example \ref{p2}).
This completes the proof of Theorem \ref{main1}.
\end{proof}

We prepare two lemmas for the proof of Theorem \ref{main2}.
Since $\Box_G$ is a flag polytope, we obtain the following lemma.

\begin{lemma}\label{link}
Let $G$ be a finite simple graph on $V(G)=\{1, \ldots, n\}$
and let $N \in \mathcal{N}^\Box(G)$ with $|N|=n-1$.
Then there exists a pair $\{J, J'\} \subset (\mathcal{I}_G
\cup \{\{\overline{1}\}, \ldots, \{\overline{n}\}\}) \setminus N$ such that
$N \cup \{J\}$ and $N \cup \{J'\}$ are distinct maximal elements of $\mathcal{N}^\Box(G)$.
Furthermore, $\{J, J'\}$ is not in $\mathcal{N}^\Box(G)$.
\end{lemma}

The proof of the following lemma
is the same as a part of the proof of \cite[Theorem 3.4]{Suyama},
but we describe it for the reader's convenience.

\begin{lemma}\label{cyclediamond}
Let $G$ be a finite simple graph on $V(G)=\{1, \ldots, n\}$
and let $J, J' \in \mathcal{I}_G$ such that $J \cap J' \ne \emptyset$
and $G|_{J \cap J'}$ is not connected.
Then there exists $I \subset V(G)$ such that
$G|_I$ is a cycle with $\geq 4$ nodes or a diamond graph.
\end{lemma}

\begin{proof}
Let $G|_{I_1}, \ldots, G|_{I_m}$ be the connected components of $G|_{J \cap J'}$.
We pick $x \in I_1, x' \in I_2 \cup \cdots \cup I_m$
and take simple paths $x=y_1, y_2, \ldots, y_r=x'$ in $G|_J$
and $x=z_1, z_2, \ldots, z_s=x'$ in $G|_{J'}$. Let
\begin{align*}
p&=\mathrm{max}\{i \in \{1, \ldots, r\} \mid y_i \in I_1,
y_i=z_j \mbox{ for some } j \in \{1, \ldots, s\}\},\\
q&=\mathrm{min}\{i \in \{p+1, \ldots, r\} \mid y_i \in I_2 \cup \cdots \cup I_m,
y_i=z_j \mbox{ for some } j \in \{1, \ldots, s\}\}.
\end{align*}
Then we have two simple paths between $y_p$ and $y_q$.
The two paths have no common nodes except $y_p$ and $y_q$.
Since $y_p \in I_1$ and $y_q \in I_2 \cup \cdots \cup I_m$,
we have $\{y_p, y_q\} \notin E(G)$ and the number of edges of
each path is greater than or equal to two.
Thus we obtain a simple cycle of length $\geq 4$ containing $y_p$ and $y_q$.
Hence we may assume that there exist integers $k$ and $l$
with $3 \leq k<l \leq n$
such that $\{1, 2\}, \{2, 3\}, \ldots, \{l-1, l\}, \{l, 1\} \in E(G)$
and $\{1, k\} \notin E(G)$.
We may further assume that $\{i, j\} \notin E(G)$ for every
\begin{itemize}
\item $1 \leq i<j \leq k$ where $j-i \geq 2$,
\item $k \leq i<j \leq l$ where $j-i \geq 2$,
\item $k \leq i \leq l-1$ and $j=1$,
\end{itemize}
since if such an edge exists,
then we can replace the cycle by a shorter cycle containing the edge.
We find $I \subset V(G)$ such that
$G|_I$ is a cycle with $\geq 4$ nodes or a diamond graph as follows:

{\it The case where $\{2, l\} \notin E(G)$}. We consider
\begin{align*}
i_\mathrm{min}&
=\mathrm{min}\{i \in \{2, \ldots, k\} \mid
\{i, j\} \in E(G) \mbox{ for some } j \in \{k+1, \ldots, l\}\},\\
j_\mathrm{max}&
=\mathrm{max}\{j \in \{k+1, \ldots, l\} \mid \{i_\mathrm{min}, j\} \in E(G)\}.
\end{align*}
Then the induced subgraph by the subset
\begin{equation*}
\{1, 2, \ldots, i_\mathrm{min}, j_\mathrm{max}, j_\mathrm{max}+1, \ldots, l\} \subset V(G)
\end{equation*}
is a cycle with $\geq 4$ nodes (see Figure \ref{cycle2}).
\begin{figure}[htbp]
\begin{center}
\includegraphics[width=12cm]{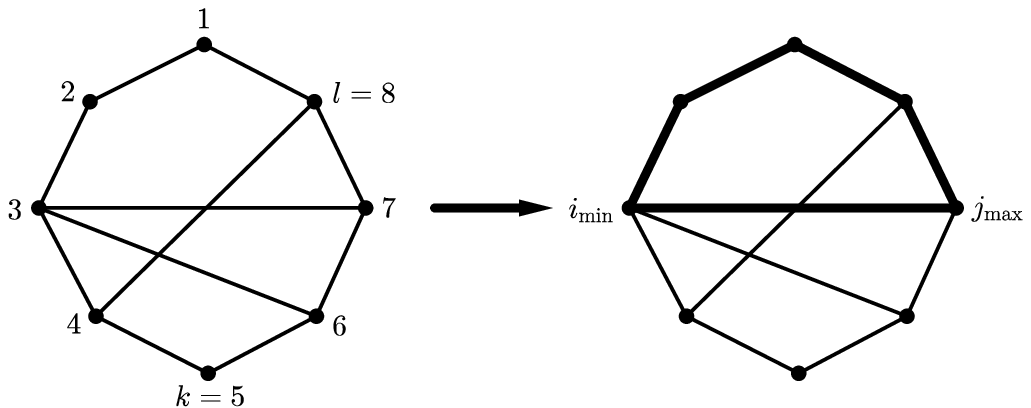}
\caption{a cycle as an induced subgraph.}
\label{cycle2}
\end{center}
\end{figure}

{\it The case where $\{2, l\} \in E(G)$}.
Suppose that there exists an integer $j$
such that $k+1 \leq j \leq l-1$ and $\{2, j\} \in E(G)$. We consider
\begin{equation*}
j_\mathrm{max}=\mathrm{max}\{j \in \{k+1, \ldots, l-1\} \mid \{2, j\} \in E(G)\}.
\end{equation*}
If $j_\mathrm{max}=l-1$, then the induced subgraph by the subset
$\{1, 2, l-1, l\}$ is a diamond graph.
If $j_\mathrm{max} \leq l-2$, then the induced subgraph by the subset
\begin{equation*}
\{2, j_\mathrm{max}, j_\mathrm{max}+1, \ldots, l\} \subset V(G)
\end{equation*}
is a cycle with $\geq 4$ nodes.
Suppose that $\{2, j\} \notin E(G)$ for any $k+1 \leq j \leq l-1$.
We consider
\begin{align*}
i_\mathrm{min}&
=\mathrm{min}\{i \in \{3, \ldots, k\} \mid
\{i, j\} \in E(G) \mbox{ for some } j \in \{k+1, \ldots, l\}\},\\
j_\mathrm{max}&
=\mathrm{max}\{j \in \{k+1, \ldots, l\} \mid \{i_\mathrm{min}, j\} \in E(G)\}.
\end{align*}
If $i_\mathrm{min}=3$ and $j_\mathrm{max}=l$,
then the induced subgraph by the subset $\{1, 2, 3, l\}$ is a diamond graph.
Otherwise, the induced subgraph by the subset
\begin{equation*}
\{2, 3, \ldots, i_\mathrm{min}, j_\mathrm{max}, j_\mathrm{max}+1, \ldots, l\} \subset V(G)
\end{equation*}
is a cycle with $\geq 4$ nodes.

Thus we obtain $I \subset V(G)$ such that
$G|_I$ is a cycle with $\geq 4$ nodes or a diamond graph.
This completes the proof.
\end{proof}

\begin{proof}[Proof of Theorem \ref{main2}]
We assume $V(G)=\{1, \ldots, n\}$.

$(1) \Rightarrow (2)$:
(i) Suppose that there exists $I \subset V(G)$
such that $G|_I$ is a cycle with $\geq 4$ nodes.
We may assume that $I=\{1, \ldots, k\}, k \geq 4$ and
\begin{equation*}
E(G|_I)=\{\{1, 2\}, \{2, 3\}, \ldots, \{k-1, k\}, \{k, 1\}\}.
\end{equation*}
We consider
\begin{equation*}
N=\{\{1\}, \{1, 2\}, \ldots, \{1, \ldots, k-3\}, \{k-1\}, \{1, \ldots, k\},
\{\overline{k+1}\}, \ldots, \{\overline{n}\}\} \in \mathcal{N}^\Box(G).
\end{equation*}
Then the pair in Lemma \ref{link} is
$J=\{1, \ldots, k-1\}$ and $J'=\{1, \ldots, k-3, k-1, k\}$. Since
\begin{equation*}
e_J+e_{J'}-e_{\{1, \ldots, k-3\}}-e_{\{k-1\}}-e_{\{1, \ldots, k\}}=0,
\end{equation*}
Proposition \ref{intersection number} gives
$(-K_{X(\Delta_{\Box_G})}.V(\sigma_N))=2-3=-1$.
Therefore $X(\Delta_{\Box_G})$ is not weak Fano by Proposition \ref{Nakai} (2).

(ii) Suppose that there exists $I \subset V(G)$
such that $G|_I$ is a diamond graph.
We may assume that $I=\{1, 2, 3, 4\}$ and
\begin{equation*}
E(G|_{\{1, 2, 3, 4\}})=\{\{1, 2\}, \{1, 3\}, \{1, 4\}, \{2, 3\}, \{2, 4\}\}.
\end{equation*}
We consider
\begin{equation*}
N=\{\{3\}, \{4\}, \{1, 2, 3, 4\},
\{\overline{5}\}, \ldots, \{\overline{n}\}\} \in \mathcal{N}^\Box(G).
\end{equation*}
Then the pair in Lemma \ref{link} is
$J=\{1, 3, 4\}$ and $J'=\{2, 3, 4\}$. Since
\begin{equation*}
e_J+e_{J'}-e_{\{3\}}-e_{\{4\}}-e_{\{1, 2, 3, 4\}}=0,
\end{equation*}
Proposition \ref{intersection number} gives
$(-K_{X(\Delta_{\Box_G})}.V(\sigma_N))=2-3=-1$.
Therefore $X(\Delta_{\Box_G})$ is not weak Fano by Proposition \ref{Nakai} (2).

(iii) Suppose that there exists $I \subset V(G)$
such that $G|_I$ is a claw.
We may assume that $I=\{1, 2, 3, 4\}$ and
\begin{equation*}
E(G|_{\{1, 2, 3, 4\}})=\{\{1, 2\}, \{1, 3\}, \{1, 4\}\}.
\end{equation*}
We consider
\begin{equation*}
N=\{\{2\}, \{3\}, \{4\},
\{\overline{5}\}, \ldots, \{\overline{n}\}\} \in \mathcal{N}^\Box(G).
\end{equation*}
Then the pair in Lemma \ref{link} is
$J=\{1, 2, 3, 4\}$ and $J'=\{\overline{1}\}$. Since
\begin{equation*}
e_J+e_{J'}-e_{\{2\}}-e_{\{3\}}-e_{\{4\}}=0,
\end{equation*}
Proposition \ref{intersection number} gives
$(-K_{X(\Delta_{\Box_G})}.V(\sigma_N))=2-3=-1$.
Therefore $X(\Delta_{\Box_G})$ is not weak Fano by Proposition \ref{Nakai} (2).

$(2) \Rightarrow (1)$:
Suppose that $X(\Delta_{\Box_G})$ is not weak Fano.
By Proposition \ref{Nakai} (2), there exists $N \in \mathcal{N}^\Box(G)$
such that $|N|=n-1$ and $(-K_{X(\Delta_{\Box_G})}.V(\sigma_N)) \leq -1$.
By Lemma \ref{link},
there exists a pair $\{J, J'\} \subset (\mathcal{I}_G
\cup \{\{\overline{1}\}, \ldots, \{\overline{n}\}\}) \setminus N$ such that
$N \cup \{J\}$ and $N \cup \{J'\}$ are distinct maximal elements of $\mathcal{N}^\Box(G)$.
Furthermore, $\{J, J'\}$ is not in $\mathcal{N}^\Box(G)$.
Thus we must have $J \in \mathcal{I}_G$ or $J' \in \mathcal{I}_G$.
We may assume $J \in \mathcal{I}_G$.

{\it The case where $J' \in \mathcal{I}_G$}.
We have $J \not\subset J', J' \not\subset J$ and $J \cup J' \in \mathcal{I}_G$.
We will show that $J \cap J' \ne \emptyset$
and $\{J \cup J'\} \cup (\mathcal{I}_{G|_{J \cap J'}})_{\rm max} \subset N$,
where $(\mathcal{I}_{G|_{J \cap J'}})_{\rm max}$
is the set of maximal (by inclusion) elements of $\mathcal{I}_{G|_{J \cap J'}}$.

We show $N \cup \{J, J \cup J'\} \in \mathcal{N}^\Box(G)$,
which implies $J \cup J' \in N$.
Since $\Box_G$ is flag, it suffices to show that the two facets
$F_I$ and $F_{J \cup J'}$ intersect for every $I \in N$.
Suppose $\{\overline{i}\} \in N \cap \{\{\overline{1}\}, \ldots, \{\overline{n}\}\}$.
Since $\{\{\overline{i}\}, J\}, \{\{\overline{i}\}, J'\} \in \mathcal{N}^\Box(G)$,
we have $i \notin J$ and $i \notin J'$, so $i \notin J \cup J'$.
Thus $F_{\{\overline{i}\}}$ and $F_{J \cup J'}$ intersect.
It remains to show that $F_I$ and $F_{J \cup J'}$ intersect
for every $I \in N \cap \mathcal{I}_G$.
Since $\{I, J\}, \{I, J'\} \in \mathcal{N}^\Box(G)$,
we see that $I$ falls into the following three cases:
\begin{itemize}
\item $I \subset J$ or $I \subset J'$. Then $I \subset J \cup J'$.
\item $I \cup J \notin \mathcal{I}_G$ or $I \cup J' \notin \mathcal{I}_G$.
We may assume $I \cup J \notin \mathcal{I}_G$.
If $J' \subset I$, then $I \cup J=(I \cup J') \cup J=I \cup (J \cup J') \in \mathcal{I}_G$,
which is a contradiction. Thus $I \subset J'$ or $I \cup J' \notin \mathcal{I}_G$.
If $I \subset J'$, then $I \subset J \cup J'$.
Suppose $I \cup J' \notin \mathcal{I}_G$.
If $I \cup (J \cup J') \in \mathcal{I}_G$,
then at least one of $G|_{I \cup J}$ and $G|_{I \cup J'}$ is connected,
which is a contradiction.
Thus $I \cup (J \cup J') \notin \mathcal{I}_G$.
\item $J \subset I$ and $J' \subset I$. Then $J \cup J' \subset I$.
\end{itemize}
In every case, $F_I$ and $F_{J \cup J'}$ intersect.
Hence $N \cup \{J, J \cup J'\} \in \mathcal{N}^\Box(G)$.
Since $|N \cup \{J, J \cup J'\}| \leq n$, we must have $J \cup J' \in N$.

If $J \cap J'=\emptyset$,
then we have $e_J+e_{J'}-e_{J \cup J'}=0$
and Proposition \ref{intersection number} gives
$(-K_{X(\Delta_{\Box_G})}.V(\sigma_N))=2-1=1$, which is a contradiction.
Hence $J \cap J' \ne \emptyset$.
Let $C \in (\mathcal{I}_{G|_{J \cap J'}})_{\rm max}$.
We show $N \cup \{J, C\} \in \mathcal{N}^\Box(G)$,
which implies $C \in N$.
It suffices to show that
$F_I$ and $F_C$ intersect for every $I \in N$.
Suppose $\{\overline{i}\} \in N \cap \{\{\overline{1}\}, \ldots, \{\overline{n}\}\}$.
Since $\{\{\overline{i}\}, J\} \in \mathcal{N}^\Box(G)$,
we have $i \notin J$, so $i \notin C$.
Thus $F_{\{\overline{i}\}}$ and $F_C$ intersect.
It remains to show that $F_I$ and $F_C$ intersect
for every $I \in N \cap \mathcal{I}_G$.
Since $\{I, J\}, \{I, J'\} \in \mathcal{N}^\Box(G)$,
we see that $I$ falls into the following three cases:
\begin{itemize}
\item $J \subset I$ or $J' \subset I$. Then $C \subset J \cap J' \subset I$.
\item $I \cup J \notin \mathcal{I}_G$ or $I \cup J' \notin \mathcal{I}_G$.
If $I \cup C \in \mathcal{I}_G$,
then $I \cup J=I \cup(C \cup J)=(I \cup C)\cup J \in \mathcal{I}_G$
and $I \cup J'=I \cup(C \cup J')=(I \cup C)\cup J' \in \mathcal{I}_G$,
which is a contradiction.
Thus $I \cup C \notin \mathcal{I}_G$.
\item $I \subset J$ and $I \subset J'$. Then $I \in \mathcal{I}_{G|_{J \cap J'}}$.
Since $C \in (\mathcal{I}_{G|_{J \cap J'}})_{\rm max}$,
we have $I \subset C$ or $I \cap C=\emptyset$.
If $I \cap C=\emptyset$,
then $C \subsetneq I \cup C \subset J \cap J'$, so $I \cup C \notin \mathcal{I}_G$.
\end{itemize}
In every case, $F_I$ and $F_C$ intersect.
Hence $N \cup \{J, C\} \in \mathcal{N}^\Box(G)$.
Since $|N \cup \{J, C\}| \leq n$, we must have $C \in N$.
Therefore $\{J \cup J'\} \cup (\mathcal{I}_{G|_{J \cap J'}})_{\rm max} \subset N$.

We see that
\begin{equation*}
e_J+e_{J'}-\sum_{C \in (\mathcal{I}_{G|_{J \cap J'}})_{\rm max}}e_C-e_{J \cup J'}=0.
\end{equation*}
Proposition \ref{intersection number} gives
$-1 \geq (-K_{X(\Delta_{\Box_G})}.V(\sigma_N))
=2-|(\mathcal{I}_{G|_{J \cap J'}})_{\rm max}|-1
=1-|(\mathcal{I}_{G|_{J \cap J'}})_{\rm max}|$.
Thus $|(\mathcal{I}_{G|_{J \cap J'}})_{\rm max}| \geq 2$,
that is, $G|_{J \cap J'}$ is not connected.
By Lemma \ref{cyclediamond}, there exists $I \subset V(G)$ such that
$G|_I$ is a cycle with $\geq 4$ nodes or a diamond graph.

{\it The case where $J'=\{\overline{j}\}$ for $j \in \{1, \ldots, n\}$}.
Since $\{J, J'\} \notin \mathcal{N}^\Box(G)$, we have $j \in J$.
By Proposition \ref{intersection number}, we have the relation
\begin{equation*}
e_J+e_{J'}+\sum_{I \in N}a_Ie_I=0, \quad a_I \in \mathbb{Z}.
\end{equation*}
Let $N'=\{I \in N \cap \mathcal{I}_G \mid I \subset J \setminus \{j\}\}$
and let $I_1, \ldots, I_r$ be the maximal elements of $N'$.
Since $N \cup \{J\}, N \cup \{J'\} \in \mathcal{N}^\Box(G)$,
we have $I \subset J \setminus \{j\}$ or $I \cap J=\emptyset$
for every $I \in N \cap \mathcal{I}_G$,
and $i \notin J$ for every
$\{\overline{i}\} \in N \cap \{\{\overline{1}\}, \ldots, \{\overline{n}\}\}$.
The relation above implies that for each $k \in J \setminus \{j\}$,
there exists $I \in N'$ such that $I \ni k$ and $a_I<0$.
Hence we have $J \setminus \{j\}=I_1 \cup \cdots \cup I_r$.
Since $I_1, \ldots, I_r$ are pairwise disjoint, the relation above is
\begin{equation*}
e_J+e_{J'}-e_{I_1}-\cdots-e_{I_r}=0.
\end{equation*}
Proposition \ref{intersection number} gives
$-1 \geq (-K_{X(\Delta_{\Box_G})}.V(\sigma_N))=2-r$, so $r \geq 3$.

We show that $G|_{I_p \cup \{j\}}$ is connected for every $p=1, \ldots, r$.
We pick $x \in I_p$ and take a simple path $x=x_1, \ldots, x_s=j$ in $G|_J$.
Then $x_i \notin I_p$ for some $i \in \{1, \ldots, s\}$.
Let $m=\mathrm{min}\{i \in \{1, \ldots, s\} \mid x_i \notin I_p\}$.
If $x_m \in I_q$ for some $q \ne p$, then $I_p \cup I_q \in \mathcal{I}_G$,
which contradicts that $\{I_p, I_q\} \subset N \in \mathcal{N}^\Box(G)$.
Thus $x_m$ must be $j$.
Since $x_1, \ldots, x_s$ is a simple path, we must have $m=s$.
Hence $\{x_{s-1}, j\} \in E(G)$. We put $y_p=x_{s-1}$.
Since $\{y_p, y_q\} \notin E(G)$ for $p \ne q$,
the induced subgraph $G|_{\{j, y_1, y_2, y_3\}}$ is a claw.

In every case, we obtain a desired induced subgraph.
This completes the proof of Theorem \ref{main2}.
\end{proof}

\end{document}